\documentclass[10pt]{amsart}
\usepackage{amsmath}
\usepackage{latexsym}
\usepackage{amsfonts}
\usepackage{amssymb}

\theoremstyle{plain}
\newtheorem{theorem}{Theorem}

\newtheorem{lemma}[theorem]{Lemma}
\newtheorem{proposition}[theorem]{Proposition}
\newtheorem*{thm}{Theorem}

\theoremstyle{definition}
\newtheorem{definition}[theorem]{Definition}

\newcommand{\essi}{\operatornamewithlimits{ess\,inf}}
\newcommand{\esss}{\operatornamewithlimits{ess\,sup}}

\newcommand{\Om}{\Omega}
\newcommand{\lb}{\lambda}

\newcommand{\dif}{\mathrm{d}}

\title[Kolmogorov compactness criterion in variable \ldots]
{Kolmogorov compactness criterion \\
in variable exponent Lebesgue spaces}
\author[H. Rafeiro]{Humberto Rafeiro}
\address{Universidade do Algarve\\ Departamento de Matem\'atica\\
 Campus de Gambelas 8005-139 Faro -- PORTUGAL}
\email{hrafeiro@ualg.pt}


\date{}

\begin{document}
\maketitle

\begin{abstract}
The well-known Kolmogorov compactness criterion is extended to the case of
variable exponent Lebesgue spaces $L^{p(\cdot)}(\overline{\Omega})$, where
$\Omega$ is a bounded open set in $\mathbb R^n$ and $p(\cdot)$ satisfies some
``standard'' conditions. Our final result should be called
 Kolmogorov-Tulajkov-Sudakov compactness criterion, since it includes the
 case $p_-=1$ and requires only the ``uniform'' condition.
\end{abstract}

\section{Introduction}
We extend Kolmogorov-Tulajkov-Sudakov criterion to the case of variable
exponent Lebesgue spaces $L^{p(\cdot)}(\Om)$. The theory of the variable
exponent Lebesgue spaces   was intensively developed during the last two
decades, inspired both by difficult open problems in this theory, and
possible applications shown in \cite{525}, we refer e.g.  to the surveying
papers \cite{106b, 316b, 580bd} on the topic  and references therein.

The classical theorem of Kolmogorov \cite{kol}(see also, e.g., \cite{nat})
about compactness of subsets in $L^p$, can be stated in the following terms
\begin{thm}[\textbf{Kolmogorov}]
Suppose  $\mathfrak F$ is a set of functions in $L^p\left([0,1]\right)(1<p<\infty)$. In order that this set be relatively compact,
it is necessary and sufficient that both of the following conditions be satisfied:
\begin{enumerate}\addtolength{\itemsep}{0.5\baselineskip}
\item[(k1)] the set $\mathfrak F$ is bounded in $L^p$; \item[(k2)]
$\lim\limits_{h\to 0}\|f_h-f\|_p=0$ uniformly with respect to $f\in\mathfrak
F$,
\end{enumerate}
where $f_h$ denotes the well-known Steklov function
\[
f_h(x)= \frac{1}{2h} \int_{x-h}^{x+h} f(t) \;\dif t .
\]
\end{thm}
 After that, Tamarkin \cite{tam} extended the result to the case where the
 underlying space can be unbounded, with an additional condition related to
 behavior at infinity. Tulajkov \cite{tul} showed that Tamarkin's result was
 true even when $p=1$. Finally, Sudakov \cite{sud} showed that condition (k1)
 follows from condition (k2).

Kolmogorov's compactness criterion has also been extended to other function spaces, for example, it was extended by Takahashi \cite{tak} for \textit{Orlicz spaces} satisfying the $\Delta_2$-condition, by Goes and Welland \cite{goewel} for \textit{continuously regular K\"othe spaces} and by Musielak \cite{musbook} to \textit{Musielak-Orlicz spaces}, just to mention a few.

The main result of this paper is given in Theorems \ref{kolmogorov} and
\ref{kolmogorov1}.
 Although it seems that we can obtain our Theorems
\ref{kolmogorov} and \ref{kolmogorov1} from \cite{goewel}, we give its
straightforward proof, which is easier than to derive it from the multi-step
proof in \cite{goewel}. The direct proof we suggest is completely within the
frameworks of variable exponent Lebesgue spaces.

A version of the proof of Theorem \ref{kolmogorov} that appears in \cite[p.
63]{musbook} in the context of Musielak-Orlicz spaces, uses an extra
condition, the so-called \textit{$\infty$-condition} (see \cite[p.
61]{musbook}), but this condition is not satisfied even in the case of
constant $p=1$, while our proof admits the case $\inf p(x)=1$.

\section{Preliminaries}

We refer to papers
\cite{332,  575a} and surveys \cite{106b, 316b, 580bd}   for details on variable Le\-bes\-gue
spaces over domains  in $\mathbb R^n$, but  give some necessary definitions.
  Let $\Om\subseteq \mathbb{R}^n$ be an open set in $\mathbb{R}^n$. By
$L^{p(\cdot)}(\Om)$ we denote the space of functions
$f(x)$ on $\Om$ such that
\[    I_{p}(f)=\int_{\Om}|f(x)|^{p(x)}\;\dif x\]
where  $p(x)$ is a measurable function on $\Om$  with values in
$[1,\infty)$ and define $p_-=\essi_{x\in\Om}p(x)$ and $p_+=\esss_{x\in\Om}p(x)$.
This is a Banach  space with respect to the norm
\begin{equation}\label{f:2xsep5d}
    \|f\|_{L^{p(\cdot)}(\Om)}=\inf\left\{\lb>0: I_{p}\left(\frac{ f}{\lb}\right)\leqq 1\right\}.
\end{equation}

We list the following properties of the space $L^{p(\cdot)}(\Omega)$ that will be needed:

\begin{itemize}
\item[(a)]\textit{H\"older's inequality}:
\begin{equation}\label{holder}
\int_\Omega |f(x)g(x)|\;\dif x \leqq k \|f\|_{p(\cdot)} \|g\|_{q(\cdot)}
\end{equation}
where $1\leqq p(x)\leqq \infty$, $\frac{1}{p(x)}+\frac{1}{q(x)}\equiv 1$, $k=\sup\limits_{x\in
\Omega}\frac{1}{p(x)}+ \sup\limits_{x\in \Omega}\frac{1}{q(x)}$. \item[(b)] \textit{estimates for
the norm of the characteristic function of a set}:
\begin{equation}\label{estcha}
|E|^{1/p_+}\leqq \|\chi_E\|_{p(\cdot)}\leqq |E|^{1/p_-},\quad \mathrm{if}\; |E|\leqq 1 ,
\end{equation}
the signs of the inequalities being opposite if $|E|\geqq 1$. \item[(c)]\textit{denseness of step
functions}: in the case $p_+<\infty$, functions of the form $\sum_{k=1}^m c_k \chi_{\Omega_k}$,
$\Omega_k\subset \Omega$, $|\Omega_k|<\infty$, with constant $c_k$, form a dense set in
$L^{p(\cdot)}(\Omega)$. \item[(d)]\textit{denseness of continuous functions}: in the case
$p_+<\infty$,  the set of continuous functions with a finite support is dense in
$L^{p(\cdot)}(\Omega)$. \item[(e)]$L^{p(\cdot)}(\Omega)$ \textit{is ideal}: i.e. it is complete and
the inequality $|f|\leqq |g|$ implies $\|f\|_{p(\cdot)} \leqq \|g\|_{p(\cdot)}$.
\end{itemize}

Standard conditions in the framework of variable Lebesgue space are
\begin{equation}\label{log}
|p(x)-p(y)|\leqq \frac{C}{-\ln|x-y|}, \quad |x-y|\leqq \frac12 ,
\end{equation}
 and
\begin{equation}\label{decay}
|p(x)-p(\infty)|\leqq \frac{C}{\ln(2+|x|)}.
\end{equation}

\begin{definition}
By $w$-$Lip(\Omega)$ we denote the class of exponents $p\in L^\infty(\Omega)$ satisfying the (local) logarithmic condition \eqref{log}.
\end{definition}

\begin{definition}
 By $\mathcal P_\infty(\Omega)$ we denote the class of exponents $p\in L^\infty (\Omega)$ which satisfy the condition that there exists $p(\infty)=\lim\limits_{\Omega \ni x \to \infty} p(x)$ and the assumption \eqref{decay}.
\end{definition}

\subsection{Approximate identities}

Let $\phi$ be an integrable function which have $\int_{\mathbf R^n} \phi(x)\;\dif x =1$. For each $t>0$,
we put $\phi_t:=t^{-n}\phi(t^{-1}x)$. Following \cite{urifio07}, we say that $\{\phi_t \}$ is a
potential-type approximate identity, if the radial majorant of $\phi$, defined by
$$\widetilde{\phi}(x)=\sup_{|y|\geq|x|} |\phi(y)|$$
is integrable. In \cite{urifio07} the following proposition was proved.

\begin{proposition}\label{urifio}
Given an open set $\Omega$, let $p\in\mathcal{P}_\infty(\Om)\cap w$-$Lip(\Omega)$. If  $\{\phi_t\}$
is a potential-type approximate identity, then for all $t>0$, we have:

\begin{itemize}\addtolength{\itemsep}{0.5\baselineskip}
\item[\textnormal{(i)}]$\| \phi_t \ast f   \|_{p(\cdot)}  \leqq C \| f   \|_{p(\cdot)}$, 
\item[\textnormal{(ii)}] $\displaystyle \lim_{t \to 0} \| \phi_t \ast f -f
\|_{p(\cdot)}=0 .$
\end{itemize}
\end{proposition}

The following propositions are well-known, and can be found in standard books dealing with
the subject of compactness, they are stated for the sake of completeness.

\begin{definition}
Let $(X,d)$ be a metric space, $X_0$ a subspace of $X$, and $\varepsilon>0$.
A subset $N$ in $X$ is called a \textit{Hausdorff} $\varepsilon$\textit{-net} (or just $\varepsilon$\textit{-net}) for $X_0$ if
for each element $x\in X_0$ there is a $x_\varepsilon \in N$ such that $d(x,x_\varepsilon)<\varepsilon$.
\end{definition}

\begin{definition}
A subset $X_0$ of a metric space $X$ is called \textit{totally bounded} if for every $\varepsilon>0$ in $X$ there is a
finite (i.e., consisting of finitely many elements) $\varepsilon$-net for $X_0$.
\end{definition}

\begin{proposition}\label{hausdorff}
A subset of a complete metric space is relatively compact (i.e., its closure is compact) if and only if it is totally bounded.
\end{proposition}

\begin{definition}
Let $X$ and $Y$ be metric spaces. A family $\mathcal F$ of functions $f:X \to Y$ is \textit{equicontinuous
on} $X$ if for every $\varepsilon>0$ there exists a $\delta >0$ such that $d_Y(f(x_1),f(x_2))<
\varepsilon$ for any function $f\in \mathcal F$  and all $x_1,x_2 \in X$ such that $d_X(x_1,x_2)<\delta$.
\end{definition}

\begin{definition}
The functions of a set $\mathcal F $ are said to be \textit{uniformly bounded} if
$\displaystyle\sup_{f\in \mathcal F} \sup_{x\in X} |f(x)|<\infty.$
\end{definition}

\begin{proposition}[\bf{Ascoli-Arzel\`a}]\label{AscArz}
Let $\Omega$ be a bounded open set in $\mathbb R^n$. A subset $K$ of $C(\overline{\Omega})$ is relatively compact  in $C(\overline{\Omega})$ if $K$ is equicontinuous and is also uniformly bounded.
\end{proposition}

\section{Compactness criterion}
\subsection{Steklov function}
We will use both $f_h$ or $\Phi_h\ast f$ to denote the Steklov function
\begin{equation}\label{steklov}
f_h(x)=\Phi_h\ast f(x)=\frac{1}{v_nh^n }\int_{B(x,h)} f(t)\;\dif t
\end{equation}
where $h>0, \ \Phi(x)=\frac{1}{v_n}\chi_{B(0,1)}(x)$, $\Phi_h(x)=
\frac{1}{h^n}\Phi\left(\frac{x}{h}\right)$,\ $
v_n=|B(0,1)|=\frac{\pi^{\frac{n}{2}}}{\Gamma\left(\frac{n}{2}+1\right)},$ and it is assumed that
the function $f(x)$ is continued beyond $\Om$ as identical zero, whenever necessary.

\begin{lemma}\label{steklovappro}Given an open set $\Omega$,
 let $p\in\mathcal{P}_\infty(\Om)\cap w$-$Lip(\Omega)$. Then
\begin{itemize}\addtolength{\itemsep}{0.5\baselineskip}
\item[\textnormal{(i)}] $\|f_h\|_{p(\cdot)}\leqq C\|f\|_{p(\cdot)};$
\item[\textnormal{(ii)}] $\lim\limits_{h\to 0} \|f_h-f  \|_{p(\cdot)}=0$.
\end{itemize}
where $C>0$ does not depend on $h>0$ and $f\in L^{p(\cdot)}(\Om).$
\end{lemma}
\begin{proof}
It suffices to note  that $\{\Phi_h\}$ is a potential-type approximate identity and apply
Proposition \ref{urifio}.
\end{proof}

\subsection{Kolmogorov compactness criterion}The following result is an extension to the variable Lebesgue spaces of the well known Kolmogorov compactness criterion (or, to be more precise, Kolmogorov-Tulajkov compactness criterion since we admit $p_-=1$).

\begin{theorem}\label{kolmogorov}
Let $\mathfrak F$ be a subset of $L^{p(\cdot)}(\overline\Omega)$, where $\Omega$ is a bounded open set in
 $\mathbb R^n$ and let $p\in\mathcal{P}_\infty(\overline\Omega)\cap w$-$Lip(\overline\Omega)$. The set $\mathfrak F$ is relatively compact if and only if the
following conditions are satisfied:
\begin{itemize}\addtolength{\itemsep}{0.5\baselineskip}
\item[\textnormal{(i)}]$ \lim\limits_{h \to 0} \| f_h -f \|_{p(\cdot)}=0$ uniformly for $f \in \mathfrak F$,
\item[\textnormal{(ii)}]$\mathfrak F$ is bounded in $L^{p(\cdot)}(\overline\Omega)$.
\end{itemize}
\end{theorem}

\begin{proof}
\textit{Necessity}.

We only have to prove the equicontinuity (i).
 By Proposition
\ref{hausdorff}, for every
$\varepsilon >0$, there exists a finite $\varepsilon$-net for the set $\mathfrak F$. Since step functions are dense in $L^{p(\cdot)}(K)$, the finite
$\varepsilon$-net  can be chosen as the set $\{s_j\}_{j=1}^\ell$ of simple functions $s_j(x)$ with
the property
\begin{equation}\label{density}
\|f-s_j\|_{p(\cdot)}<\varepsilon.
\end{equation}
To prove (i), we first note that, by Lemma \ref{steklovappro}(ii), given $\varepsilon >0$, there
exists an  $h_j$ indexed to each $s_j$ such that
\[
\| \Phi_h\ast s_j -s_j \|_{p(\cdot)} < \varepsilon
\]
whenever $h<h_j$. Letting $h_0 =\min_{1\leqq j\leqq \ell} h_j$, we have
\[
\| \Phi_h\ast s_j -s_j\|_{p(\cdot)}   < \varepsilon
\]
for all $j=1,\ldots,\ell$ whenever $h<h_0$.

Then for $h<h_0$ and all $f \in \mathfrak F$ we have a suitable $s_r$ such that
\[
\begin{split}
\| \Phi_h \ast f-f \|_{p(\cdot)} &\leqq   \|\Phi_h \ast f-\Phi_h \ast s_r \|_{p(\cdot)} + \|s_r-f\|_{p(\cdot)}+\|\Phi_h \ast s_r-s_r  \|_{p(\cdot)}\\
& \leqq (C+1) \|f-s_r\|_{p(\cdot)}+\varepsilon=(C+2)\varepsilon
\end{split}
\]
where $C$ is from Lemma \ref{steklovappro}, which gives the necessity of (i).

\vspace{3mm}
\textit{Sufficiency}.

Let  $\mathfrak F_h=\{f_h:f\in \mathfrak F\}$, where $f_h$ is the Steklov
function \eqref{steklov}. By H\"older's inequality \eqref{holder} and property \eqref{estcha}, we obtain
\begin{equation}\label{steuni}
\begin{split}
v_n h^n|f_h(x)|& \leqq k\|f \|_{p(\cdot)} \|\chi_{B(x,h)} \|_{q(\cdot)}\\
& \leqq kM(v_n h^n)^{1/q_-}
\end{split}
\end{equation}
whenever $v_nh^n\leqq 1$. This means that all the functions in $\mathfrak F_h$ are uniformly
bounded under suitable choice of $h$ due to condition (ii).

Let us define another set, namely, $\mathfrak F_{hh}=\{f_h:f\in \mathfrak F_h\}$. Functions of $\mathfrak F_{hh}$ are of the form
\begin{equation}
f_{hh}(x)=\frac{1}{v_nh^n}\int_{B(x,h)}f_h(t)\;\dif t.
\label{fhh}
\end{equation}
By \eqref{steuni} and using the same considerations made above, we have that all the functions in $\mathfrak F_{hh}$ are uniformly bounded.

We want also to show that they are equicontinuous. We have
\begin{equation}\label{equi}
\begin{split}
v_nh^n|f_{hh}(x+u)-f_{hh}(x)|&\leqq \int_{\Omega(x,u,h)}|f_h(t)|\;  \dif t\\
&\leqq kM(v_n h^n)^{1/q_--1} |\Omega(x,u,h)|
\end{split}
\end{equation}
where $\Omega(x,u,h)=B(x+u,h)\Delta B(x,h)$ and $\Delta$ stands for the symmetric difference of sets.

Since  $B(x+u,h)\subset B(x,|u|+h)$ and $B(x+u,h)\supset B(x,h-|u|)$ whenever $|u|\leqq h$, we have
that
\begin{equation}\label{sets}
\Theta(x,u,h):=B(x,|u|+h)\backslash B(x,h) \cup  B(x,h)\backslash B\left(x,h-|u|\right)\supset \Omega(x,u,h).
\end{equation}
The Lebesgue measure of $\Theta$ is given by
\begin{equation}\label{theta}
\begin{split}
|\Theta(x,u,h)|&=v_n\left \{ [(|u|+h)^n-h^n]+[h^n -(h-|u|)^n]  \right\}\\
&\leqq  2 n v_n|u| (2h)^{n-1},\quad \mathrm{when}\; |u|\leqq h.
\end{split}
\end{equation}
By \eqref{equi}-\eqref{theta} we have the inequality
\[|f_{hh}(x+u)-f_{hh}(x)|\leqq C |u|\]
whenever $|u| \leqq h$ and $h$ is  fixed and sufficiently small, which proves the
equicontinuity of $\mathfrak F_{hh}$.

Thus, for an arbitrary fixed $h>0$ such that $v_nh^n\leqq 1$, the set $\mathfrak F_{hh}$ is
relatively compact in $C(\overline\Omega)$ due to Ascoli-Arzel\`a proposition \ref{AscArz}.
By the fact that $C(\overline\Omega)$ is dense in $L^{p(\cdot)}(\overline\Omega)$ we can take an $\varepsilon/2$ approximation by continuous functions, and from this $\varepsilon/2$ approximation we can obtain a finite $\varepsilon/2$-net of $\mathfrak F_{hh}$ to $C(\overline\Omega)$, entailing  a finite $\varepsilon$-net of $\mathfrak F_{hh}$ to $L^{p(\cdot)}(\overline\Omega)$, which gives relative compactness by Proposition \ref{hausdorff}.
 Finally, relative compactness of $\mathfrak F$ regarding $L^{p(\cdot)}(\overline\Omega)$ follows from condition (i), relative compactness of $\mathfrak F_{hh}$ and the same reasoning as above applied twice, first to $\mathfrak F_h$ and finally to $\mathfrak F$.
\end{proof}

\vspace{3mm}

The condition (ii) of Theorem \ref{kolmogorov}  can be omitted, since it
follows from condition (i), this was proved in \cite{sud} for the case of
classical Lebesgue spaces, which is also valid for the case of variable
exponent Lebesgue spaces, as shown in Lemma \ref{sudakov}.

To this end, we need the following results, where Proposition \ref{nik} is a
statement known in the Riesz-Schauder theory, see e.g. \cite[pp.
283-286]{yosida}

\begin{proposition}\label{nik}
Let $U$ be a linear operator with a compact power. If $\lambda_0 \neq 0$ is not an eigenvalue of $U$, then $\lambda_0$ is in the resolvent set of $U$, i.e., the operator $(U-\lambda_0 I)^{-1}$ is continuous.
\end{proposition}

\begin{lemma}
Let $G$ be a subset of a normed space and, for some $K>0$, we have that for all $x\in G\; \| Ux-x\|<K$. Then for such $x$
\[
\|x\|\leq K \|(U-I)^{-1}\|
\]
i.e., the set $G$ proves to be bounded whenever $(U-I)^{-1}$ is a bounded operator.
\label{apr}
\end{lemma}

We can now prove the following lemma.

\begin{lemma}\label{sudakov}
Let $\mathfrak F$ be a subset of $L^{p(\cdot)}(K)$, where $K$ is a bounded
closed set in $\mathbb R^n$ and let $p\in\mathcal{P}_\infty(K)\cap
w$-$Lip(K)$. Then condition $\mathrm{(i)}$ of Theorem \ref{kolmogorov}
implies condition $\mathrm{(ii)}$.
\end{lemma}

\begin{proof}
By Proposition \ref{nik} and Lemma \ref{apr}, if $\lambda =1$ is not an eigenvalue of $U$, the boundedness of $G$ is a consequence of a uniform approximation of elements  $x\in G$ by elements of $Ux$ whenever $U$ has some power of it being compact.

In our case, the operator $U$ is given by the Steklov function, i.e., $Ux=\Phi_h\ast x$, it is compact in square in $L^{p(\cdot)}(K)$ (see the sufficiency part in the proof of Theorem \ref{kolmogorov}) and the uniform approximation property is given by condition (i) of Theorem \ref{kolmogorov}.
 Therefore we only need to  show that $\lambda=1$ cannot be an eigenvalue of the Steklov operator.

Let us suppose, for the sake of contradiction, that exists $L^{p(\cdot)}(K)\ni x(t) \neq 0$ such that
\begin{equation}\label{equ}
\Phi_h \ast x(t)=x(t)\quad \textrm{for}\quad t\in K,
\end{equation}
 where we continue the function $x(t)$ by zero beyond $K$.

 The function $\Phi_h \ast x(t)$ is continuous and vanishes beyond some bounded set. Being continuous attains its maximum and minimum and at the least one of them is different from $0$.

Let us suppose that $M=\max_{t} \Phi_h \ast x(t)>0$.
We define $P$ has the set on $\mathbb R^n$ such that $\Phi_h x(t)$ attains its maximum. $P$ is closed and bounded.
Let us choose an arbitrary boundary point $t_0 \in \partial P \cap K$.
Then the ball $B(t_0,h)$ contains a set of positive measure in $K$ such that $x(t)=\Phi_h\ast x(t)<M$, this by \eqref{equ} and the definition of $P$,
but then $\Phi_h \ast x(t_0)$ cannot attain the maximum since it is an average of $x(t)$ inside $B(t_0,h)$. The obtained contradiction proves the statement.
\end{proof}

By Theorem \ref{kolmogorov} and Lemma \ref{sudakov} we obtain

\begin{theorem}\label{kolmogorov1}
Let $\mathfrak F$ be a subset of $L^{p(\cdot)}(\overline\Omega)$, where $\Omega$ is a bounded open set in
 $\mathbb R^n$ and let $p\in\mathcal{P}_\infty(\overline\Omega)\cap w$-$Lip(\overline\Omega)$. The set $\mathfrak F$ is relatively compact if and only if the
following condition is satisfied:
\[ \lim\limits_{h \to 0} \| f_h -f \|_{p(\cdot)}=0 \quad \textrm{uniformly for}\quad f \in \mathfrak F.\]
\end{theorem}

\section*{Acknowledgements}
I would like to thank Prof. S. Samko for  helpful suggestions during the
preparation of  the paper.
I would also like to acknowledge the financial support of
\textit{Funda\c c\~ao para a Ci\^encia e a Tecnologia} (FCT) through grant
Nr. SFRH/ BD/ 22977 / 2005 of the Portuguese Government.

\end{document}